\documentclass[A4paper,12pt]{article}
\usepackage{latexsym}
\usepackage{mathrsfs}
\usepackage{amssymb}
\usepackage{amscd}
\usepackage[dvips]{graphicx}                  

\newtheorem{theorem}{Theorem}
\newtheorem{corollary}{Corollary}
\newtheorem{lemma}{Lemma}
\newtheorem{proposition}{Proposition}

\newenvironment{proof}[1][Proof]{\textbf{#1.} }{\ \rule{0.5em}{0.5em}}
\long\def\symbolfootnote[#1]#2{\begingroup%
\def\thefootnote{$\;$}\footnote[#1]{$^*$#2}\endgroup}
\begin{document}

\title{A note on asymptotic density}
\author{Ryszard Frankiewicz and Joanna Jureczko\footnote{corresponding author}}

	\maketitle
	
	\footnote{
		
		\hspace{0.2cm}
		}
	
	\begin{abstract}	It is proved that $P(\omega)/\triangle_d$, where $\triangle_d$ is the ideal of sets of asymptotic density zero, is universal in the sense of embeddings.
	\end{abstract}
	
	\section{Introduction}
	
	We examine the algebra $P(\omega)/\triangle_d$, where $\triangle_d$ is the ideal of sets of asymptotic density zero. We start this paper with Proposition 1, in which we show that  $P(\omega)/ fin$ is embedded into $P(\omega)/\triangle_d$.
	Thus, one can suspect that $P(\omega)/\triangle_d$ has a number of common features with the algebra $P(\omega)/fin$. 
	Actually, so, because, for example both  algebras have no $(\omega, \omega)$-gaps, (compare \cite[p. 38]{FZ} and Proposition 2 in this paper). They are also different in some aspects, i.e., $P(\omega)/fin$ has no $\omega$-limits, (see \cite[p.33]{FZ}) but, as we will show in Proposition~3, $P(\omega)/\triangle_d$ has.
	Moreover, as we will show in Proposition 4, $(P(\omega)/\triangle_d)^\omega $ is isomorphic to $P(\omega)/\triangle_d$, the same property is not true for $P(\omega)/fin$ even for the finite product, (see e.g.\cite{GJ} p. 89 and 215 for details).
	
	Furthermore, the space $\omega^* = st(P(\omega)/fin)$ was recognized as homogeneous up to 1956 when  W. Rudin proved that under CH it has $P$-points, (see \cite{R} and \cite[p. 79]{FZ}). In 1978  K. Kunen proved in ZFC that there exist weak P-points  in $\omega^*$, (see \cite{K1}), more precisely, he showed  that in $\omega^*$ there are weak $P$-points which are not $P$-points. (Later in \cite{S} there was proved that in some models $P$-points  do not exist in $\omega^*$).  As we will show in Theorem~1, $P(\omega)/\triangle_d$ has a weak $P$-point, (our argumentation is based on \cite{K, K1}, where there is constructed a special kind of family, namely an independent linked family).

	By well known Parovi\v cenko theorem,  (see \cite[p.43-45]{FZ}),  $P(\omega)/fin$ is  universal in the sense of embeddings of algebras of cardinality continuum.
	However, as was shown by Shelah in \cite{S} the algebra $LM/\triangle$, (where $LM$ is the collection of all Lebesgue measurable sets and $\triangle $ is the ideal of sets of Lebesgue measure zero) consistently does not have to be embedded in $P(\omega)/fin$.
	However, it is true that after adding $\omega_2$ Cohen reals one can embed $LM/\triangle$ into $P(\omega)/fin$, but it is not equivalent to CH, (see \cite{CFZ}). In the presence of Proposition 1, we can conclude that there are such models for which there is no embedding $P(\omega)/ \triangle_d$ into $P(\omega)/fin$. Despite this result,
	we show that after adding $\omega_2$ Cohen reals  $ P(\omega)/\triangle_d$ can be embedded in  $ P(\omega)/fin$, (see Theorem 2 in which the proof we use the ideas presented in~\cite{CFZ}).
	
	In \cite{F} it is proved that there is a measure preserving embedding from $LM/\triangle$ into $P(\omega) /\triangle_d$.
	This statement leads us to the most spectacular result of this paper which we include in Theorem 3, i.~e.
	there exists an embedding $\phi \colon H(2^{\mathfrak{c}})/ \triangle_H \to P(\omega)/\triangle_d$ such that $d(\phi(a)) =~\mu(a),$ $a \in H(2^{\mathfrak{c}})/ \triangle_H$,  where $H(2^{\mathfrak{c}})$ is the family of all Haar measurable sets of $2^{\mathfrak{c}}$, $\mu$ is the Haar measure on the Cantor cube $2^{\mathfrak{c}}$ and $\triangle_H$ is the ideal of sets of Haar measure zero.
	(In the proof of Theorem 3 we use only combinatorial methods and ideas presented in paper \cite{F}).
	
	The paper is organized in two sections. In Section 2, there are gathered basic definitions and notations used in the further part of this paper. Section 3 is divided into four subsections containing: embeddings into $P(\omega) /\triangle_d$, the existence of a weak $P$-point in $st(P(\omega) /\triangle_d)$, the embeddings $P(\omega) /\triangle_d$ into $P(\omega) /fin$ (under assumptions mentioned above), the measure preserving embedding  $H(2^{\mathfrak{c}})/ \triangle_H$ into $P(\omega) /\triangle_d$.

	\section{Definitions and previous results}
	
	\textbf{1.1.} An ordinal $\omega$  denotes the set of non-negative integers,  $fin$  denotes the ideal of all finite subsets of $\omega$.
	\\
	\\
	\textbf{1.2.} Let $A \subseteq \omega$. \textit{Lower asymptotic density} $d_* (A)$   and  \textit{upper asymptotic density} $d^* (A)$  is defined respectively 
	$$d_*(A) = \lim_{n \to \infty} \inf \frac{|A\cap n|}{n}\ \textrm{ and }\  d^*(A) =\lim_{n \to \infty} \sup \frac{|A\cap n|}{n}.$$
	If
	$d_*(A)= d^*(A) = :d(A)$, then $d(A)$ is called \textit{(asymptotic) density}. 
	By $\triangle_d$ we denote the ideal of subsets of (asymptotic) density zero.
	Notice that sets with  density do not form an algebra, but in many aspects they behave as an algebra. In cases where density may not  exist we will  consider the upper density.
	\\
	\\
	\textbf{1.3.} Let $\mathbb{A}$ and $\mathbb{B}$ be Boolean algebras. We say that $\phi \colon \mathbb{A} \to \mathbb{B}$ is a \textit{regular embedding} iff
	\\
	(I) $\phi$ is homomorphism, 
	\\
	(II) $\phi$ is embedding,
	\\
	(III) if $\mathcal{A}$ is a maximal antichain in $\mathbb{A}$, then $\{\phi(A) \colon A \in \mathcal{A}\}$ is a maximal antichain in $\mathbb{B}$,
	\\
	(see also \cite[vol. 1]{MKB}). 
	\\
	\\
	\textbf{1.4.}
	Let $\mathbb{B}$ be a Boolean algebra.
	A \textit{gap} (or \textit{cut}) is a pair 
	$\langle S, T \rangle$, where $S, T \subseteq \mathbb{B}$ such that  $a\cdot b = \mathbb{O}$ for each $a \in S$ and $b \in T$.
	The element $c$ \textit{fills  (or separates)} a gap $\langle S, T \rangle$ if $a\leq c$ and $c \cdot b = \mathbb{O}$ for each $a \in S$ and $b \in T$.
	A pair of cardinals $(\kappa, \lambda)$, where $\kappa = |S|$ and $\lambda = |T|$ is called \textit{a type} of $\langle S, T \rangle$ if there is no $c\in \mathbb{B}$ which separates $\langle S, T \rangle$. The system without separating elements is called a  \textit{$(\kappa, \lambda)$-gap}.
	\\
	\\
	\textbf{1.5.}
	A descending chain $\{a_\alpha \colon \alpha < \gamma\}$ in Boolean algebra $\mathbb{B}$ by \textit{a $\gamma$-limit (a limit of length $\gamma$)}, if there is no positive $a$ such that $a < a_\alpha$ for each $\alpha < \gamma$.
	\\
	\\
	\textbf{1.6.}
	Let $\mathbb{B}$ be a Boolean algebra. The \textit{gap determined by $c$ over $\mathbb{B}$} is a pair $\langle A, B \rangle$, where $A$ consists of all elements of $\mathbb{B} $ with $a \leq c$ and $B$ consists of all elements of $\mathbb{B}$ with $c \leq b$.     
	$\langle A, B \rangle$ is \textit{countably generated} if $A$ is a countably generated as an ideal and $B$ is countably generated as a filter (possibly improper). 
	\\
	\\
	\textbf{1.7.}
	The $st(\mathbb{A})$  denotes the Stone space of $\mathbb{A}$. Then $\beta \omega = st(P(\omega))$ and  $\beta \omega \setminus \omega = st(P(\omega)/fin) = \omega^*$, (see \cite{CN, M} for details).
	\\
	\\
	\textbf{1.8.} Let $X$ be a topological space. A point $p \in X$ is called a \textit{$P$-point} iff  $p \in \textrm{int} (\bigcap_n U_n)$, whenever $\{U_n \colon n \in \omega\}$ is a countable family of neighbourhoods of $p$. A point $p$ is called a \textit{weak $P$-point} iff $p$ is not a limit point of any countable subset of $X \setminus \{p\}$. Clearly, every $P$-point is a weak $P$-point, but not conversely, (see \cite{K1}).
	\\
	\\
	\textbf{1.9.}
	$LM$  denotes the family of all Lebesgue measurable sets in $2^\omega$,  $\lambda$  denotes the Lebesgue measure  and   $\triangle$  denotes the ideal of subsets of Lebesgue measure zero.
	\\
	\\
	\textbf{1.10.}
	$H(2^{\kappa})$ denotes the family of all Haar measurable sets of $2^{\kappa}$,   where $\kappa$ is an infinite cardinal. The Haar measure $\mu$ is generated by the basic sets $[s] = \{f \in \{0, 1\}^{\kappa} \colon s \subseteq f\}$, where $s \in \{0, 1\}^{<\omega}$ and 
	$\mu([s]) = \frac{1}{2^{|s|}}$. The ideal  of sets of Haar measure zero is denoted by $\triangle_H$.
	
	Notice that we can consider $2^{\mathfrak{c}}$ as the product space $(2^\omega)^{\mathfrak{c}}$. 
	If we endow the Cantor cube $2^{\mathfrak{c}}$ with the structure of topological group with Haar measure $\mu$ then there is a continuous map 
	$\varphi \colon 2^{\mathfrak{c}} \to 2^\omega$ satisfying $\varphi(\mu) = \lambda$, where $\lambda$ is the Lebesgue measure on $2^\omega$, (it is folklore).
	Then all theorems concerning the base of such a product and product measure are applied (see \cite{B, E}).
	\\
	\\
	\textbf{1.11.} For definitions not cited here, we refer the reader to \cite{B, CN, E, FZ, GJ, M, MKB, SS}.

	\section{Main results}
	
	\subsection{}
	
	\begin{proposition}
		There exists a regular embedding $P(\omega)/ fin$ into $P(\omega)/\triangle_d$.
	\end{proposition}
	
	\begin{proof}
		Let $\{k_n\colon n \in \omega\}$ be a sequence  of elements of $\omega$ such that $k_n<~k_{n+1},$ $ n \in \omega$ and 
		$$\lim_{n \to \infty} \frac{k_{n+1} - k_{n}}{k_{n+1}} = 1.$$
		Let $f \colon P(\omega)/fin \to P(\omega)/\triangle_d$ be a function such that
		$$f([A]) = [\bigcup_{n \in A} [k_n, k_{n+1})],$$ 
		where $[A] \in P(\omega)/fin$ and $[\bigcup_{n \in A} [k_n, k_{n+1})]] \in P(\omega)/\triangle_d$. 
		Clearly, $f$ is a homomorphism and even embedding, because $d(\bigcup_{n \in A} [k_n, k_{n+1}))>0$.
		If $\mathcal{A}$ is a maximal antichain in $P(\omega)/fin$, then it is clear  that $\{f([A]) \colon A \in \mathcal{A}\}$ is a maximal antichain in $P(\omega)/\triangle_d$.
	\end{proof}
	\\
	
	\begin{proposition}
		$P(\omega) /\triangle_d$ has no $(\omega, \omega)$-gaps.
	\end{proposition}
	
	\begin{proof}
		Consider a gap in $P(\omega)/\triangle_d$
		$$\langle \{[A_n] \colon n \in \omega\}, \{[B_m] \colon m \in \omega\} \rangle,$$
		such that 
		\\
		(i) $d^*(A_n) >0, d^*(B_n)>0$ for $n, m \in \omega$,
		\\
		(ii) $A_\alpha \cap B_\beta =_*0$ for any $\alpha, \beta \in \omega$, i.e. $|A_\alpha \cap B_\beta| < \omega$. 
		\\Then $d(A_\alpha \cap B_\beta) = 0$ for any $\alpha, \beta \in \omega$.
		
		Now for each $n \in \omega$ take $C_n \subseteq \omega$ such that
		$\{A_n \setminus C_n \colon n \in \omega\}$ is an increasing sequence. Since property AP holds, (see \cite{MA}), there is $C \subseteq \omega$ such that
		\\
		(iii) $d^*((A_n\setminus C_n) \setminus C) = 0$
		\\
		(iv) $d^*(C) = \lim_{n \to \infty} d^*(A_n \setminus C_n)$.
		\\
		Then for each $n \in \omega$ we have 
		$$d^*(C \cap B_m) = d^*((A_n \setminus C) \cup (A_n \cap B_m)) = d^*(A_n \setminus C)+ d^*(A_n \cap B_m) = 0.$$ The proof is completed
	\end{proof}
	\\

	\begin{proposition}
		$P(\omega)/ \triangle_d$ has an $\omega$-limit.
	\end{proposition}
	
	\begin{proof}
		Consider a chain $\{a_n \colon n \in \omega\}$ such that
		\\
		(i) $a_n \geqslant a_{n+1}$,
		\\
		(ii) $a_n = [A_n] \in P(\omega)/\triangle_d$,
		\\
		(iii) $d(A_n) =\frac{1}{2^n}$.
		\\
		Observe that $A_n \supseteq A_{n+1}$ and $\bigcap_{n\in \omega} A_n  = 0$.
		Then $d(\bigcap_{n\in \omega} A_n ) =0$ and $[\bigcap_{n\in \omega} A_n ] =\mathbb{O}$.
	\end{proof}
	\\

	\begin{proposition}
		There exists an embedding $f\colon (P(\omega)/\triangle_d)^\omega \to P(\omega)/\triangle_d$.
	\end{proposition}	
	
	\begin{proof}
		Consider a sequence $\{k_n \colon n \in \omega\}$ of elements of $\omega$ such that $k_n <~k_{n+1},$ $ n \in \omega$ and
		$$\lim_{n \to \infty} \frac{k_{n+1} - k_n}{k_{n+1}} = 1.$$
		Consider a partition of $\omega$ into infinite sets of positive density $\{A_n \colon n \in \omega\}$ such that $k_n \in A_n$ and $k_m \not \in A_n$ for any distinct $m, n \in \omega$. Hence $[A_n]\cap [A_m] = \emptyset$, thus the family $\{P(A_n)/\triangle_d \colon n \in \omega\}$ is a partition of $P(\omega) /\triangle_d$.
		Obviously the mapping
		$$h_n \colon P(\omega)/\triangle_d \to P(A_n)/\triangle_d^n,$$
		where $\triangle^n_d$ is the ideal of subsets of $A_n$ of density zero,
		is an isomorphism, because $A_n$ is infinite for any $n \in \omega$.
		Then
		$$(h_n)_{n \in \omega} = h \colon (P(\omega)/\triangle_d)^\omega \to \Pi_{n \in \omega} P(A_n)/\triangle_d^n$$
		is also an isomorphism.
		
		Consider a mapping
		$$g \colon P(\omega)/\triangle_d \to \Pi_{n \in \omega} P(A_n)/\triangle_d^n$$
		such that 
		$g(x) = (x \cdot a_n)_{n\in \omega}$, where $a_n \in P(A_n)/\triangle_d^n$. 
		Notice that $g$ is homomorphism because $h_n = \varphi \circ g$, where $\varphi: ( P(\omega)/\triangle_d)^\omega \to P(\omega)/\triangle_d$ is a projection. The mapping $g$ is also one-to-one because if $x \not = \mathbb{O}$ in $P(\omega)/\triangle_d$, then $x \cdot a_n \not = \mathbb{O}$ for some $n \in \omega$.
		Now take $b = (b_n)_{n \in \omega} \in \Pi_{n \in \omega} P(A_n)/\triangle_d^n$ and $x = \sum_{n \in \omega} b_n$. Then $g(x) = b$, because $\{P(A_n)/\triangle_d \colon n \in \omega\}$ is a partition. Hence $g$ is an isomorphism.
		Then $$f = g^{-1} \circ h \colon (P(\omega)/\triangle_d)^\omega \to P(\omega)/\triangle_d$$ is the required embedding.			 
	\end{proof}

	\subsection{}
	
	\begin{lemma}
		For any family $\mathcal{A} = \{A_n \subseteq \omega \colon n \in \omega\}$ and  $T \subset \omega$ be a set of indices such that $|T| < \omega$ 
		$$d^*(\bigcup_{n \in T}A_n) \leq \sum_{n \in T} d^*(A_n).$$
		Moreover if $\Sigma_{n \in \omega}d^*(A_n) < \infty$, then
		$$d^*(\bigcup_{n \in \omega}A_n) \leq \sum_{n \in \omega} d^*(A_n).$$
	\end{lemma}
	
	\begin{proof}
		Enumerate $T = \{0, ..., m\}$ for some natural number $m$ and take $\{A_n \subseteq \omega \in n \in T\}$.
		Then by standard argumentation we have
		$$\frac{|(\bigcup_{m \in T} A_m)\cap n|}{n} = \frac{|(A_{0}\cap n)\cup ...\cup (A_m\cap n)|}{n} \leq \frac{|A_0 \cap n|}{n} + ... + \frac{|A_m \cap n|}{n}.$$
		Taking the limit supremum under $n \to \infty$ we obtain the first part of the lemma.
		
		Assuming $\Sigma_{n \in \omega}d^*(A_n) < \infty$ and using standard analytical arguments, we obtain the second part of the lemma.
	\end{proof}

	\begin{lemma}
		Let $\{A_n \subseteq \omega \colon d^*(A_n) \leqslant \frac{1}{n+1}\}$ and let $S \subset \omega$ be a set of indices such that $|S| < \omega$. Then 
		$d^*(\bigcup_{n \in S} A_n)\leqslant 1$
		and $\{A_n \colon n \in S\}$ is $|S|-$centered, i.e. $\bigcap_{n \in S} A_n \not = \emptyset$.
		Moreover if   $\Sigma_{n \in \omega}d^*(A_n) < \infty$,
		then $d^*(\bigcup_{n \in \omega} A_n)< 1$ .
	\end{lemma}  
	
	\begin{proof}
		Assume that $|S| = m+1$ for some $m \in \omega$. 
		By Lemma 1 
		$$d^*(\bigcup_{n  \in S}A_n) \leq \sum_{n \in S} d^*(A_n) \leqslant (m+1) \cdot \frac{1}{m+1} = 1.$$
		To verify that $\{A_n  \colon n \in S\}$ is $|S|$-centered  notice that $d^*(\omega) = 1$ and use the argumentation above.
		
		For the second part of the lemma, notice that 
		$\sum_{n \in \omega}\frac{1}{2^{n}}$ is the geometric series, use the inequality above, the fact that $d^*(\omega) =1$  and the standard analytical arguments. 
	\end{proof}
	
	\begin{lemma}
		For each ultrafilter $\mathcal{U} \in st (P(\omega)/\triangle_d)$ and for each $\varepsilon > 0$ there exists $A \in \mathcal{U}$ such that $d^*(A) < \varepsilon$.
	\end{lemma}
	
	\begin{proof}
		We start with $d^*(\omega) =1.$ Divide $\omega$ into two sets $A_0, A_1$ in this way that 
		$d^*(A_0) = d^*(A_1) = \frac{1}{2}$ and then $A_0$ divide into two sets $A_{00}, A_{01}$ in this way that 
		$d^*(A_{00}) = d^*(A_{01}) = \frac{1}{2^2}.$
		In the similar way divide $A_1$ obtaining 
		$d^*(A_{10}) = d^*(A_{11}) = \frac{1}{2^2}$.  
		Continue this procedure up to $f \in {^k}\{0,1\},$ for some $k \in \omega$.
		Let $$\mathcal{A}_k = \{A_f \subseteq \omega \colon d^*(A_f) = \frac{1}{2^k}, f \in {{^k}\{0, 1\}}\}.$$
		Then divide each $A_f \in \mathcal{A}_k$  into two sets $A_{f^\frown 0}, A_{f^\frown 1}$ in this way that
		$d^*(A_{f^\frown 0}) = d^*(A_{f^\frown 1}) = \frac{1}{2^{k+1}}$, where $f^\frown 0, f^\frown 1 \in {^{k+1}\{0, 1\}}$. Since $k$ is arbitrary, the proof is complete.	
	\end{proof}
	\\
	
	\begin{lemma}
		Let $\{A_n \colon n \in \omega\}$ be a family such that $\Sigma_{n \in \omega}d^*(A_n) \leq 1$. Then for every utrafilter  $\mathcal{U} \in st(P(\omega)/\triangle_d)$ there exists  an element $B \in \mathcal{U}$ such that $\bigcup_{n \in \omega} A_n \cap~B \not =~\emptyset$ and $d^*(\bigcup_{n \in \omega} A_n \cap B) <1$. 
	\end{lemma}
	
	\begin{proof}
		The case when $\sum_{n \in \omega} d^*(A_n) = 1$ is obvious.
		Let $\sum_{n \in \omega} d^*(A_n) < 1$.  By Lemma 1, we have $d^*(\bigcup_{n \in \omega} A_n) <1.$ Then $0<d^*(\omega\setminus \bigcup_{n \in \omega} A_n) \leqslant 1.$
		For each $n \in \omega$ take
		$$\mathcal{B}_n = \{B \in \mathcal{U} \colon A_n \cap B \not = \emptyset\}.$$
		For completeness the proof it is enough to show that there exists $n_0$ such that $\mathcal{B}_{n_0} \not  =  \emptyset$.
		If $\mathcal{B}_n = \emptyset$ for all $n \in \omega$ then all elements of $\mathcal{U}$ would be in $\omega \setminus \bigcup_{n \in \omega} A_n$. But $d^*(\bigcup \mathcal{U}) = 1$. A contradiction. \end{proof}
	\\

	\begin{theorem}
		In $st(P(\omega)/\triangle_d)$ there exists a weak P-point.
	\end{theorem}
	
	\begin{proof}	
		Consider a family
		$$\mathcal{A} = \{A'_n \subset \omega \colon d^*(A'_n) = \frac{1}{2^{n+1}}, n \in \omega\}.$$
		Such a family is infinite by Lemma 3. 
		For each $n \in \omega$ consider a family 
		$$\mathcal{A}_n = \{A^n = A'_n\setminus B \colon B \subseteq \omega  \textrm { and  } d^*(B) < \frac{1}{2^{n+1}}\}.$$
		Each $\mathcal{A}_n$ has cardinality continuum. Order each $\mathcal{A}_n$, i.e.
		$$\mathcal{A}_n = \{A^n_\alpha \colon \alpha < 2^\omega\}$$
		and construct the family 
		$$\mathcal{C}_0 = \{\bigcup_{n \in \omega}A^{n}_{\alpha_n} \colon (\alpha_n) \in [2^\omega]^\omega\}.$$
		Since the construction above and Lemma 2, $\mathcal{C}_0$ is a filter which contains only infinite subsets of $\omega$.
		By Lemma 4, $\mathcal{C}_0$ is not a required weak P-point.
		Notice that in $st(P(\omega)/ \triangle_d)$ there are $2^\mathfrak{c}$ ultrafilters but they are generated by $2^\omega$  subsets of $\omega$. Enumerate all subsets of $\omega$ by
		$\mathcal{V} = \{V_\alpha \colon \alpha < 2^\omega\}.$ 
		By Lemma 3, for each $V \in \mathcal{V}$ there exists a countable decreasing sequence
		$$V = V^0 \supseteq V^1 \supseteq V^2\supseteq ...$$
		such 
		$d^*(V^i) \geq d^*(V^j)$ for $i>j$ and $\lim_{n \to \infty} d^*(V^n) = 0$.  
		Thus we have to construct an increasing sequence of filters $\{\mathcal{C}_\alpha \in P(\omega)/\triangle_d  \colon \alpha < 2^\omega\}$ with the following properties:
		\\
		(i) $\forall_{V \in \mathcal{V}} \exists_{\alpha < 2^\omega}\ V \in \mathcal{C}_\alpha \textrm{ or } \omega \setminus V \in \mathcal{C}_\alpha $,
		\\
		(ii) $\forall_{V \in \mathcal{V}} \exists _{B \in \bigcup \mathcal{C}_\alpha} \exists_{V'\subseteq V}\ ( d^*(V') \leq d^*(V) \Rightarrow V' \cap B = \emptyset)$. 
		\\	
		Then $\mathcal{C} = \bigcup_{\alpha \in 2^\omega} \mathcal{C}_\alpha$ will be the required ultrafilter.
		
		We will proceed by transfinite induction in $2^\omega$ stages. 
		Let $$\{D^{\lambda}_{\mu, n} \colon \mu, \lambda < 2^\omega, 1 \leq n < \omega\}$$ be an independent linked family mod $fin$, see \cite{K1}.
		We will construct an increasing sequence of filters $\{\mathcal{C}_\alpha \colon \alpha < 2^\omega\}$ of properties (i)-(ii) and a decreasing sequence of indices $\{M_\alpha \colon \alpha < 2^\omega\}$ of properties
		\\
		(iii) $|M_\alpha \setminus M_{\alpha+1}| < \omega$
		\\
		(iv) 	$\{D^{\lambda}_{\mu, n} \colon \mu, \lambda < 2^\omega, 1 \leq n < \omega, \lambda \in M_\alpha\}$ is independent linked family mod~$\mathcal{C}_\alpha$.
		
		We start with the filter $\mathcal{C}_0$ constructed at the beginning of this proof and $M_0 = 2^\omega$.
		For nonzero limit ordinals $\alpha$ we define $\mathcal{C}_\alpha = \bigcup_{\beta< \alpha }\mathcal{C}_\beta$ and $M_\alpha = \bigcap_{\beta < \alpha }M_\beta$.	
		Now we will construct $\mathcal{C}_{\alpha+1}$ and $M_{\alpha+1}$. To do this, consider three cases.
		
		Case 1) $V_\alpha \not \in \mathcal{C}_\alpha$. Then there exist $C \in \mathcal{C}_\alpha$, $0<d^*(C)<1$, $S \in [M_\alpha ]^{< \omega}$, $n_\lambda \in \omega,\  \sigma_\lambda\in [2^\omega]^{n_\lambda}$ such that 
		$$V_\alpha \cap C \cap \bigcap_{\lambda\in S}(\bigcap_{\mu \in \sigma_{\lambda}} D^{\lambda}_{\mu, n_{\lambda}}) = \emptyset.$$
		Then we take $M_{\alpha+1} = M_\alpha \setminus S$ and $\mathcal{C}_{\alpha+1} = \langle\mathcal{C}_\alpha,  \bigcap_{\lambda\in S}(\bigcap_{\mu \in \sigma_\lambda} D^{\lambda}_{\mu, n_\lambda})\rangle$, (where $C = \langle A, B\rangle$ means that $C$ is generated by $A$ and $B$).
		Then obviously $\omega \setminus V_\alpha \in \mathcal{C}_{\alpha+1}$.
		
		Case 2) $V_\alpha \in \mathcal{C}_\alpha$ and there exists $n_0 \in \omega$ such that $V^{n_0}_{\alpha} \not \in \mathcal{C}_\alpha$. 	
		Then there exist $C \in \mathcal{C}_\alpha$, $0<d^*(C)<1$, $S \in [M_\alpha ]^{< \omega}, n_\lambda \in \omega, \sigma_\lambda\in [2^\omega]^{n_\lambda}$ such that 
		$$V^{n_0}_\alpha \cap C \cap \bigcap_{\lambda\in S}(\bigcap_{\mu \in \sigma_{\lambda}} D^{\lambda}_{\mu, n_{\lambda}}) = \emptyset.$$
		Then we take $M_{\alpha+1} = M_\alpha \setminus S$ and $\mathcal{C}_{\alpha+1} = \langle\mathcal{C}_\alpha,  \bigcap_{\lambda\in S}(\bigcap_{\mu \in \sigma_\lambda} D^{\lambda}_{\mu, n_\lambda}) \rangle$.
		Then obviously $\omega \setminus V^{n_0}_\alpha \in \mathcal{C}_{\alpha+1}$.
		
		Case 3)  $V_\alpha \in \mathcal{C}_\alpha$ and $V^{n}_{\alpha} \in \mathcal{C}_\alpha$ for all $n \in \omega$.
		For each $n \in \omega$ take $W^{n}_{\alpha} = \omega \setminus V^{n}_{\alpha}$.
		
		Case 3a) there exists $n_0 \in \omega$ such that $W^{n_0}_{\alpha} \not \in \mathcal{C}_\alpha$. Consider a cut $((W^{n}_{\alpha} \cap W^{n_0}_{\alpha} \colon n \in \omega), (V^{n}_{\alpha} \colon n \in \omega))$. By Proposition 2 there exists $B_\alpha \subseteq \omega$ such that $W^{n}_{\alpha} \cap W^{n_0}_{\alpha} \subseteq^* B_\alpha$ and $V^{n}_{\alpha} \cap B_\alpha = ^* \emptyset$.
		Then we take $M_{\alpha+1} = M_\alpha \setminus \{\lambda_{n_0}\}$ and 
		$\mathcal{C}_\alpha = \langle\mathcal{C}_\alpha, B_\alpha \cup \bigcup_{n< n_0}(D^{\lambda_{n_0}}_{\alpha, n}\cap W^{n_0}_{\alpha})\rangle$.
		
		Case 3b) $V^{n}_{\alpha} \in \mathcal{C}_\alpha$ and there exists $B^n \subseteq \omega$ such that $W^{n}_{\alpha} \cap B^n \in \mathcal{C}_\alpha$ for all $n \in \omega$.
		Then by Lemma 3, the construction of $\mathcal{C}_0$  and previous steps there exists $n_0 \in \omega$ such that $W^{N_0}_{\alpha}\cap B^{n_0} \in \mathcal{C}_\alpha$ and $W^{n_0}_{\alpha}\cap(\omega \setminus B^{n_0}) \not \in \mathcal{C}_\alpha$.
		Now take $\hat V^{n_0}_{\alpha} = V^{n_0}_{\alpha} \cup W^{N_0}_{\alpha}\cap B^{n_0} \in \mathcal{C}_\alpha$ and $\hat W = \omega \setminus \hat V^{n_0}_{\omega} \not \in \mathcal{C}_\alpha$.
		Consider a cut
		$$((\hat V^{n_0}_{\alpha} \cup V^{n}_{\alpha} \colon n \in \omega), (\omega\setminus (\hat V^{n_0}_{\alpha} \cup V^{n}_{\alpha}) \colon n \in \omega)).$$
		By Proposition 2 there exists $B_\alpha \subseteq \omega$ such that 
		$\omega\setminus (\hat V^{n_0}_{\alpha} \cup V^{n}_{\alpha}) \subseteq^*B_\alpha$ and $\hat V^{n_0}_{\alpha} \cup V^{n}_{\alpha} = ^*\emptyset$. Then we take
		$M_{\alpha+1} = M_\alpha \setminus \{\lambda_{n_0}\}$ and 
		$$\mathcal{C}_\alpha = \langle\mathcal{C}_\alpha, B_\alpha \cup \bigcup_{n< n_0}(D^{\lambda_{n_0}}_{\alpha, n}\cap (\hat V^{n_0}_{\alpha}\cup V^{n_0}_{\alpha}))\rangle.$$	
		The proof is complete.
	\end{proof}
	\\

	\begin{corollary}
		$st(P(\omega)/fin)$ has a weak P-point.
	\end{corollary}
	
	\begin{proof}
		Using the properties of (asymptotic) density and the argumentation used in the previous proof, we obtain our claim.
	\end{proof}
	
	\subsection{}
	
	\begin{lemma}
		If $P$ is a partial ordering for adding $\kappa$ Cohen reals for some cardinal $\kappa$, then  in $V^P$  the gap determined by any element of $P(\omega)/\triangle_d$  over $(P(\omega)/\triangle_d)^V$ is countably generated.
	\end{lemma}
	
	\begin{proof}
		In the first part of the proof we consider adding one Cohen real.
		Let $V^P = V[G]$, where $G \subseteq P$ and $G$ is a $V$-generic filter. 
		Let $c~\in P(\omega)/\triangle_d$. 
		For each $g \in G$ let 
		$$c_g = \min \{a \in (P(\omega)/\triangle_d)^V \colon g \Vdash a \wedge \bar c = 0\},$$
		where $\bar c$ is the name for c.
		Thus $\{c_g \colon g \in G\}$ is the first component of the gap. The second one is obtained similarly. 
		
		Since in $V[G]$ every element of $P(\omega)/\triangle_d$ is in $V[x]$ for some Cohen real $x$. Now, we use the first part of the proof for each $c \in P(\omega)/\triangle_d $. Then the proof is complete.
	\end{proof}
	
	\begin{lemma}
		Let $\mathbb{B}$ be a subalgebra of $P(\omega)/\triangle_d$ such that for each $c \in P(\omega)/\triangle_d$
		the gap determined by $c$ over $\mathbb{B}$ is countably generated. 
		Let $\mathbb {C}$ is the subalgebra generated by $\mathbb{B}$  with some countable set $C \subseteq \omega$. Then the gap  determined by each $c \in P(\omega)/\triangle_d$  over $\mathbb{C}$ is countably generated. 
	\end{lemma}
	
	\begin{proof}
		Let $\langle A, D\rangle$ be a gap such that $$A = \{a \in \mathbb{B} \colon a \leqslant b\} \textrm{ and } D = \{d \in \mathbb{B} \colon b \leqslant d\}.$$ 
		Let $c \in \mathbb{C}$.
		Consider a gap 
		$\langle A', D'\rangle$  such that $$A' = \{a' \in \mathbb{B} \colon a' \leqslant b\vee (-c)\} \textrm{ and } D' = \{d' \in \mathbb{B} \colon b\vee(-c) \leqslant d\}.$$
		Take an arbitrary $a'\in A'$. Observe that $A$ is generated by all $a' \wedge c$. Since $A'$ is countably generated, hence $A$ is also. Similarly, $D$ is countably generated because $D'$ is countably generated. 
	\end{proof}
	
	\begin{lemma}(CH)
		If $P$ is the partial ordering   for adding $\omega_2$ Cohen reals, then there is enumeration of  $P(\omega)/\triangle_d$
		in $V^P$ of length $\omega_2$ such that the gap determined by each element over the subalgebra generated by the previous elements is countably generated.
	\end{lemma}
	
	\begin{proof}
		Consider the algebra $\mathbb{A}_\beta$ in the extension of $V$ by first $\beta$ Cohen reals, $|\mathbb{A}_\beta| = \omega_1$ for $\beta < \omega_2$.
		
		Let $ \{a_\alpha \colon \alpha \in \omega_2\}$ be a sequence enumerating $P(\omega) / \triangle_d$ such that $$\mathbb{A}_\beta = \{a_\alpha \colon \omega_1 \beta \leq \alpha < \omega_1 (\beta+1) \}.$$
		
		Let $\mathbb{A}$  be the algebra generated by 
		$\{a_\gamma \colon \gamma < \beta < \omega_2\}$.
		Let $\mathbb{A}_\gamma$  be the algebra generated by $a_\gamma$, with $\gamma < \omega_1 \beta$, where $\alpha = \omega_1 \beta + \eta$ for $\beta < \omega_2$ and $\eta < \omega_1$.
		
		By Lemma 6, we only need to show that for any $a \in P(\omega)/\triangle_d$ the gap determined by $a$ over $\mathbb{A}$ is countably generated. The case $\beta = 0$ is clear.
		By Lemma 5, each element of $ P(\omega)/ \triangle_d$ determines a gap over $\mathbb{A}_\eta$ and is countably generated for any $\eta$.
		
		If $cf(\beta) = \omega_1$ or $\beta = \delta + 1$ for some $\delta < \omega_2$, then $\mathbb{A}$ is $\mathbb{A}_\beta$ or $\mathbb{A}_\delta$ respectively.
		
		If $cf(\beta) = \omega$, then $\mathbb{A} = \bigcup \{ \mathbb{A}_\xi \colon \xi < \beta\}$ and the gap determined by $a$ over $\mathbb{A}$ is the union of gaps determined by $a$ over $A_\xi$, $\xi < \beta$, each countably generated. The proof is complete.
	\end{proof}
	\\
	
	\begin{theorem} 
		Assume (CH). After adding $\omega_2$ Cohen reals the natural embedding from $P(\omega)/\triangle_d$ from the ground model can be extended to embedding into $P(\omega)/fin$.
	\end{theorem}
	
	\begin{proof} 
		By Lemma 7, we can enumerate $P(\omega)/\triangle_d$ by $\{a_\beta \colon \beta \in \omega_2\}$ in this way that the gap determined by each $a_\beta$ in the subalgebra generated by $a_\alpha, \alpha < \beta$ is countably generated.
		
		By induction on $\beta$ for each algebra $\mathbb{A}_\beta$ generated by $a_\alpha, \alpha < \beta$ a continuous tower of embeddings $g_\beta$ of $\mathbb{A}_\beta$ into $P(\omega)/ fin$ will be defined.
		
		We will use  a partial order, defined by Kunen, for filling a gap in $P(\omega)/fin$, (see \cite{CFZ}). 
		Let $\langle A, B \rangle$ be a gap in $P(\omega)/fin$. Let $P(A, B)$ be the partial ordering constructing of all partial functions $p \in {^{\omega} 2}$ such that $\{n \colon p(n) = 1\}/fin$ is in the ideal generated by $A$ and $\{n \colon p(n) = 0\}/fin$ is in the filter generated by $B$.
		Observe that if $\langle A, B \rangle$ is countably generated then $P(A, B)$ has a countable dense set and a Cohen real will add a generic filter for $P(A, B)$.
		If $C \subseteq \omega$ is added by $P(A, B)$ then $C/fin$ fills the gap $\langle A, B \rangle$.
		
		Notice that there is a collection $\mathcal{A}$ of at most $\omega_1$ dense subsets of $P(A, B)$ such that any filter over $P(A, B)$ which meets all dense sets in $\mathcal{A}$ determines a subset $A \subseteq \omega$ such that 
		$g_{\beta + 1}(a_\beta) = A /fin.$
		Since $A$ and $B$ are countably generated, $P(A, B)$ has a countable dense set, so either $P(A, B)$ has an atom or it is equivalent to the forcing of adding a Cohen real. 
		Notice that there is a filter meeting any collection of at most $\omega_1$ dense sets in the partial ordering by adding a Cohen real. Hence, there exists a filter meeting each element of $\mathcal{A}$. We leave to the reader the modifications needed to construct an embedding extending a given ground model embedding. 
	\end{proof}

	\subsection{}

	\begin{lemma}
		The Cantor cube $2^{\mathfrak{c}}$ contains a countable dense subset.
	\end{lemma} 
	
	\begin{proof}
		Consider a family $\mathcal{S}$ of all finite subsets of $\omega$ For each $S \in \mathcal{S}$ consider a function $p_S \colon \omega \to \{0,1\}$ such that $p_S$ is constant on $S$ and constant on $\omega \setminus S$.
		The family $\{p_S \colon S \in \mathcal{S}\}$ is countable.
		
		We can consider $\bigcup\{2^n \colon n \in \omega\}$ instead of $2^\omega$.
		Consider a family of functions $f_S \colon 2^\omega \to \{0, 1\}$ such that $f_S$ is constant on a finite $0-1$ sequence of $2^\omega$ determined by $p_S$ and constant on its completion.
		The family $$D = \{f_S \colon S \in \mathcal{S}\}$$ is countable. 
		By Hewitt-Marczewski-Pondiczery Theorem (see, e. g. \cite{E}), the set $D$ is dense in $2^{2^{\omega}}$. 
	\end{proof}
	\\
	
	\begin{lemma}
		In  $P(\omega)/\triangle_d$ there is an independent family $\{(A^0_\alpha, A^1_\alpha)\colon  \alpha \in 2^\omega\}$ such that \\(I) $d(A^{\varepsilon}_\alpha)=\frac{1}{2}$ for each $\varepsilon \in \{0,1\}$ and $\alpha \in 2^\omega$ \\(II) $\bigcap_{s}A^{\varepsilon}_{s} = \frac{1}{2^{|s|}}$ for each  $\varepsilon \in \{0,1\}$ and $s \in [2^\omega]^{<\omega}$.
	\end{lemma}
	
	\begin{proof}
		By Lemma 8, the Cantor cube $2^{\mathfrak{c}}$ contains a countable dense set~$D$. Enumerate $D = \{d_n \colon n \in \omega\}$. By Lemma 2 and Lemma 3 (particularly the proof of Lemma 3), we can obtain a partition of $\omega$ into infinite sets indexed by element of $D$ such that
		$d^*(B_{d_n}) \leqslant \frac{1}{2^n}$ such that $B_{d_{n+1}} \subseteq \omega \setminus \bigcup_{k \leqslant n} B_{d_k}$. Thus, the mapping $D \mapsto \{B_{d_n} \colon n \in D\}$ is  continuous function. Since $\bigcup_{d_n \in D} B_{d_n} = \omega$ and by \cite[Theorem 3.6.1 and Corollary 3.6.12, p. 222]{E} we can extend $g \colon \bigcup_{d_n \in D} B_{d_n} \to D$, (in fact the mapping $\omega \mapsto 2^{\mathfrak{c}}$), to $\bar g\colon \beta \omega \to 2^\mathfrak{c}$ such that $\bar{g}(\bigcup_{d_n \in D} B_{d_n}) = D$, i.e. $\bar{g}(B_{d_n}) = \textrm{const.}$ for any $d_n \in D$.
		Notice that $h \colon \beta\omega\setminus \omega \to~2^{\mathfrak{c}}$ is onto.
		We know that $\beta\omega \setminus \omega  = st(P(\omega)/fin)$, (see \cite{M} for details).
		By Stone's Theorem, (see e.~g. \cite[Theorem 1.1, p. 63]{FZ}), we have $st(P(\omega)/fin)$ is isomorphic to $P(\omega)/fin$.
		
		Now let $\varphi_\alpha$ be a projection $2^{\mathfrak{c}}$ into $\alpha$-axe of $2^\mathfrak{c}$. Let $C^\varepsilon_\alpha =\varphi^{-1}_{\alpha}(\{\varepsilon\}) $ for any $\varepsilon \in \{0, 1\}$ and $\alpha \in 2^\omega$.
		
		Obviously $C^0_\alpha  \cap C^1_\alpha = \emptyset$ and 
		$C^0_\alpha \cup C^1_\alpha = 2^\mathfrak{c}$ for any $\alpha \in 2^\omega$ and $C^\varepsilon_\alpha \cap C^{\varepsilon'}_{\beta'} \not= \emptyset$
		for any $\varepsilon, \varepsilon ' \in \{0, 1\}$ and distinct $\alpha, \beta \in 2^\omega$. Thus, $\{(C^0_\alpha, C^1_\alpha) \colon \alpha \in 2^\omega)\}$
		is an independent family in $2^\omega$.

		Now, since $h$ is onto, then $2^{\omega} \hookrightarrow P(\omega)/fin $ is an embedding. Hence $P(\omega)/fin$ contains an independent family of cardinality continuum.
		Let $f(C^\varepsilon_\alpha)\\ = [D^\varepsilon_\alpha]$.
		By Proposition 1, there is an embedding $P(\omega)/fin \hookrightarrow P(\omega)/\triangle_d$. Fix the sequence $\{k_m \colon k \in \omega\}$ such that $\lim|[k_m, k_{m+1})| = \infty$ and $$\lim_{m \to \infty} \frac{|[k_{m}, k_{m+1})|}{k_{m+1}} =\frac{1}{2}$$
		and take mapping  $[D^\varepsilon_\alpha] \mapsto [\bigcup_{m \in D^\varepsilon_\alpha}[k_m, k_{m+1})]$ for any $\varepsilon \in \{0, 1\}$ and $\alpha \in 2^\omega$. Taking $A^\varepsilon_\alpha =  \bigcup_{m \in D^\varepsilon_\alpha}[k_m, k_{m+1})$ we obtain that  $\{(A^0_\alpha, A^1_\alpha)\colon  \alpha \in 2^\omega\}$ is the required family.
	\end{proof}
	
	\begin{theorem}
		There exists an embedding $\phi \colon H(2^{\mathfrak{c}})/ \triangle_H \to P(\omega)/\triangle_d$ such that $d(\phi(a)) = \mu(a),$ where $a \in H(2^{\mathfrak{c}})/ \triangle_H$.	
	\end{theorem}
	
	\begin{proof} By Lemma 9, there exists an independent family $\{(A^0_\alpha, A^1_\alpha) \colon \alpha \in~2^\omega)\}$ of properties (I) and (II).
		
		For any $a\in H(2^{\mathfrak{c}})/ \triangle_H$ fix exactly one $G_\delta$ set $\bar A$ such that $a = [\bar A]$ and pick a sequence $\{k_m \colon m \in \omega\}$ such that $\lim|[k_m, k_{m+1})| = \infty$ and $$\lim_{n \to \infty} \frac{|[k_{m}, k_{m+1})|}{k_{m+1}} =\frac{1}{2}.$$
		
		Let $\varphi \colon 2^\mathfrak{c} \to 2^\omega$  be a continuous map such that $\mu(\bar A) = \lambda (\varphi(\bar A))$ for any $\bar A \in 2^\mathfrak{c}$.
		
		Let \\$\mathbb{P}_{\bar A} = \{<\bar A, n, s, D>, n \in \omega, s \in [\omega]^{<\omega}, D \in [2^\omega]^{<\omega}, D \textrm{ is Boolean algebras}\}$.
		\\
		Say that $<\bar A, n, s, D> \leqslant <\bar A, n', s', D'>$ iff
		\\
		(i) $n \geqslant n', s\supseteq s', D \supseteq D'$,
		\\
		(ii) $s\cap (\sup s'+1) = s'$ and $\sup s = k_m$ for some $m \in \omega$,
		\\(iii) if there is $m' < m$ such that 
		$$\forall_{s \in D}\  \forall_{k>k_{m'}}\  \forall_{\varepsilon \in \{0,1\}} \ 
		\Big|\frac{ |A^{\varepsilon}_d\cap k|}{k} - \frac{1}{2^{|d|}}\Big| < \frac{1}{n}$$
		and $\bigcap_{d \in D'} A^{\varepsilon}_{d} \not = \emptyset$ for any finite subfamily $D' \subset D$, then
		$$\forall_{s \in D}\  \forall_{k_{m'} \leqslant k\leqslant k_{m}}\  \forall_{\varepsilon \in \{0,1\}} \ 
		\Big|\frac{ |s \cap \bigcap_{d \in D'}A^{\varepsilon}_d\cap k|}{k} - \lambda (\varphi (\bar A)\cap d)\Big| < \frac{1}{2^{|d|}n}$$
		Now consider 	$$B_d(\bar A) = \{<\bar A, n, s, D> \colon d \in D\} \textrm{ and  } 
		B_m(\bar A) = \{<\bar A, n, s, D> \colon n \geqslant m\}.$$ Both sets are dense.
		Let $G(\bar A)$ be $\{B_d(\bar A) \colon d \in 2^\omega\} \cup \{B_m(\bar A) \colon m \in \omega\}$-generic. Such $G(\bar A)$ exists because our family of subsets is dense in $\mathbb{P}_{\bar A}$ is countable.
		
		Let $$C(\bar A) = \bigcup\{s \colon <\bar A, n, s, D> \in G(\bar A)\}.$$ 
		and  
		$\{C(\bar A) \colon [\bar A] = a, a \in  H(2^{\mathfrak{c}})/ \triangle_H\}$.
		\\To complete the proof, it is enough to show that 
		\\(a) $d(C(\bar A)\cap C(\bar A')) = 0$ and $d(C(\bar A)\cup C(\bar A')) = 1$ for $a' = -a \in  H(2^{\mathfrak{c}})/ \triangle_H$, where $[A'] = a'$,
		\\(b) $d(\bigcap_{n}^{i=1}C(\bar A_i) \doteq C(\bigcap_{i=1}^{n}\bar A_i)) = 0$,
		where $\doteq$ means the symmetric difference. 
		\\ We will prove (b) only. The proof of (a) is similar.
		
		Fix $\eta > 0$. Let $D' \subset D$ be a finite family of non-comparable elements in $2^\omega$ such that
		$$\lambda (\bigcup D' \doteq C(\bar A_1\cap ... \cap \bar A_n))< \frac{\varepsilon}{3} $$
		and for any $d \in D'$
		$$|\lambda(C(\bar A_i) \cap d) - \frac{1}{2^{|d|}}|<\frac{\eta}{3}.$$
		By (iii) we have 
		$$\forall_{1\leqslant i \leqslant n}\ \forall_{\varepsilon \in \{0,1\}}\ \exists_{k^i_m}\ \forall_{k \leqslant k^i_m}\ \Big|\frac{C(\bar A_i)\cap A^{\varepsilon}_{d}\cap k}{k} - \lambda(\varphi (\bar A_i) \cap d)\Big| < \frac{\eta}{3}$$
		Since
		$$\forall_{d \in D'}\ \forall_{k > \max_i k^i_m}\ \forall_{\varepsilon \{0,1\}} \Big| \frac{C(\bar A_1)\cap ... \cap C(\bar A_n)\cap \bigcap_{d \in D'} A^\varepsilon_d\cap k}{k} - \lambda (\varphi(\bar A)\cap d)\Big| < \frac{\eta}{3}, $$ then
		for each $k>K$, whenever $K$ is sufficiently large
		$$\Big|\frac{C(\bar A_1)\cap ... \cap C(\bar A_n)\cap \bigcap_{d \in D'} A^\varepsilon_d\cap k}{k} - \frac{C(\bar A_1\cap .. \cap \bar A_n)\cap A^\varepsilon_d\cap k}{k} \Big| < \eta.$$
		The proof is complete.
	\end{proof}

	\noindent
	{\sc Ryszard Frankiewicz},
	Silesian University of Technology, Gliwice, Poland.
	\\
	{\sl e-mail: ryszard.frankiewicz@polsl.pl}
	\\
	
	\noindent
	{\sc Joanna Jureczko},
	Wroc\l{}aw University of Science and Technology, Poland
	\\
	{\sl e-mail: joanna.jureczko@pwr.edu.pl}


\begin{thebibliography}{123456}
		
		\bibitem{B} {\sc S.K. Berberian},
		{\sl Measure and Integration},
		The MacMillan Company, New York, London, 1965.
		
		\bibitem{CFZ} {\sc T. Carlson, R. Frankiewicz and P. Zbierski}, 
		{\sl Borel liftings of the measure algebra and the failure of the continuum hyphothesis},
		Proc. AMS, vol. 120(4), 1994, 1247 - 1250.
		
		\bibitem{CN} {\sc W.W. Comfort and S. Negrepontis}
		{\sl The theory of ultrafilters},
		Springer-Verlag, Berlin Heidelberg New York 1974.
		
		\bibitem{E} {\sc R. Engelking},
		{\sl General Topology},
		PWN, Warszawa 1977.
		
		\bibitem{F} {\sc R. Frankiewicz},
		{\sl Some remarks on embeddings of Boolean algebras},
		Oberwolfach, 1983, Lecture notes in Math. 1089, Springer, Berlin, 1984, 64--68.
		
		\bibitem{FZ} {\sc R. Frankiewicz and P. Zbierski},
		{\sl Hausdorff gaps and limits} Studies in Logic and foundations of mathematics, vol 132, North-Holland, 1994.
		
		
		
		
		\bibitem{GJ}{\sc L. Gillman, H. Jerison},
		{\sl Rings of Continuous Functions},
		Van Nostrand Reinhold Company, New York, Cintinnati, Toronto, Londond, Melbourn, 1960.
		
		\bibitem{K}{\sc K. Kunen},
		{\sl Some points in $\beta N$},
		Math. Proc. Cambridge Philosophical Soc, 80(3), 1976, 385-398.
		
		\bibitem{K1}{\sc K. Kunen},
		{\sl Weak P-points in $N^*$},
		Colloq. Math. Soc., J\'anos Bolyai 23, Topology, Budapest, 1978, 741-749. 
		
		
		
		\bibitem{MA}{A. H. Mekler},
		{\sl Finitely additive measures on N and the additive property},
		Proc. Amer. Math. Soc. 92  (1984), 439--444
		
		\bibitem{M}{J van Mill},
		{\sl An introduction to $\beta \omega$},
		in: Handbook of Set-Theoretic Topology,
		ed. K. Kunen and J. E. Vaughan, 
		Elsevier Science Publisher B. V., 1984.
		
		\bibitem{MKB} {\sc J. D. Monk, S. Koppelberg,  R. Bonnet},
		{\sl Handbook of Boolean algebras}, vol 1-3,  North-Holland, 1989.
		
		\bibitem{R}{\sc W. Rudin},
		{\sl Homogeneity problems in set theory of \v Cech compactcifications},
		Duke Math. J. 23, 1956, 409-419.
		
		\bibitem{S}{\sc S. Shelah},
		{\sl Lifting problems of the measure algebra},
		Israel J. Math., 45, 1983. 
		
		\bibitem{SS}{\sc S. Shelah}, 
		{\sl Proper forcing},
		Lecture Notes in Mathematics 940, Springer-Verlag, Berlin, 1982. 
	\end{thebibliography}
\end{document}